\documentclass[10pt,reqno]{amsart}
\usepackage{amsmath,amscd,amsfonts,amsthm,amssymb,latexsym,mathrsfs,graphicx,url,color,todonotes,mathtools}
\usepackage{tikz}
\usetikzlibrary{shapes}
\usetikzlibrary{calc}
\usetikzlibrary{arrows}

\usepackage{hyperref}
\hypersetup{
    bookmarks=true,         
    unicode=false,          
    pdftoolbar=true,        
    pdfmenubar=true,        
    pdffitwindow=false,     
    pdfstartview={FitH},    
    pdftitle={My title},    
    pdfauthor={Author},     
    pdfsubject={Subject},   
    pdfcreator={Creator},   
    pdfproducer={Producer}, 
    pdfkeywords={keyword1} {key2} {key3}, 
    pdfnewwindow=true,      
    colorlinks=true,       
    linkcolor=black,          
    citecolor=black,        
    filecolor=black,      
    urlcolor=black      
    }

\usepackage{hyperref, caption, subcaption}

\theoremstyle{plain}
   \newtheorem{theorem}{Theorem}[section]
   \newtheorem{proposition}[theorem]{Proposition}
   \newtheorem{lemma}[theorem]{Lemma}
   \newtheorem{corollary}[theorem]{Corollary}

\theoremstyle{definition}
   
   \newtheorem{example}[theorem]{Example}

\theoremstyle{remark}
   \newtheorem{remark}[theorem]{Remark}
   
\numberwithin{equation}{section}

\def\kk{\kern.2ex\mbox{\raise.5ex\hbox{{\rule{.35em}{.12ex}}}}\kern.2ex}

\newcommand{\midarrow}{\tikz \draw[-triangle 90] (0,0) -- +(.1,0);}

\title[Stable multivariate generalizations of matching polynomials]{Stable multivariate generalizations of matching polynomials} 

\author{Nima Amini}
\address{Department of Mathematics, Royal Institute of Technology, SE-100 44 Stockholm,
Sweden}
\email{namini@kth.se}

\begin{document}
\maketitle
\begin{abstract}
The first part of this note concerns stable averages of multivariate matching polynomials.
In proving the existence of infinite families of bipartite Ramanujan $d$-coverings, Hall, Puder and Sawin introduced the $d$-matching polynomial of a graph $G$, defined as the uniform average of matching polynomials over the set of $d$-sheeted covering graphs of $G$. We prove that a natural multivariate version of the $d$-matching polynomial is stable, consequently giving a short direct proof of the real-rootedness of the $d$-matching polynomial. Our theorem also includes graphs with loops, thus answering a question of said authors. Furthermore we define a weaker notion of matchings for hypergraphs and prove that a family of natural polynomials associated to such matchings are stable. In particular this provides a hypergraphic generalization of the classical Heilmann-Lieb theorem.
\end{abstract}

\renewcommand\theenumi{\roman{enumi}}
\section{Introduction}
The real-rootedness of the matching polynomial of a graph is a well-known result in algebraic graph theory due to Heilmann and Lieb \cite{HL}. Slightly less quoted is its stronger multivariate counterpart (see \cite{HL}) which proclaims that the multivariate matching polynomial is non-vanishing when its variables are restricted to the upper complex half-plane, a property known as \emph{stability}. Other stable polynomials occurring in combinatorics include e.g. multivariate Eulerian polynomials \cite{HV}, several bases generating polynomials of matroids (including multivariate spanning tree polynomials) \cite{COSW} and certain multivariate subgraph polynomials \cite{Wagb}.
In the present note we consider several different stable generalizations of multivariate matching polynomials. 
Hall, Puder and Sawin prove in \cite{HPS} that every connected bipartite graph has a Ramanujan $d$-covering of every degree for each $d \geq 1$, generalizing seminal work of Marcus, Spielman and Srivastava \cite{MSS,MSSb} for the case $d=2$. An important object in their proof is a certain generalization of the matching polynomial of a graph $G$, called the \emph{$d$-matching polynomial}, defined by taking averages of matching polynomials over the set of $d$-sheeted covering graphs of $G$. The authors prove (via an indirect method) that the $d$-matching polynomial of a multigraph is real-rooted provided that the graph contains no loops. We prove in Theorem \ref{thm:multidmatchpoly} that the latter hypothesis is redundant by establishing a stronger result, namely that the multivariate $d$-matching polynomial is stable for any multigraph (possibly with loops). In the final section we consider a hypergraphic generalization of the Heilmann-Lieb theorem. The hypergraphic matching polynomial is not real-rooted in general (see \cite{GZM}) so it does not admit a natural stable multivariate refinement. However by relaxing the notion of matchings in hypergraphs we prove in Theorem \ref{thm:wksubgpolystab} that an associated ``relaxed'' multivariate matching polynomial is stable. 

\section{Preliminaries}
\subsection{Graph coverings and group labelings}
In this subsection we outline relevant definitions from \cite{HPS}.
Let $G = (V(G), E(G))$ be a finite, connected, undirected graph on $[n]$. In particular we allow $G$ to have multiple edges between vertices and contain edges from a vertex to itself, i.e., $G$ is a multigraph with loops. 

A graph homomorphism $f:H \to G$ is called a \textit{local isomorphism} if for each vertex $v$ in $H$, the restriction of $f$ to the neighbours of $v$ in $H$ is an isomorphism onto the neighbours $f(v)$ in $G$. We call $f$ a \textit{covering map} if it is a surjective local isomorphism, in which case we say that $H$ \textit{covers} $G$. If the image of $H$ under the covering map $f$ is connected, then each fiber $f^{-1}(v)$ of $v \in V(G)$ is an independent set of vertices in $H$ of the same size $d$. If so, we call $H$ a \textit{$d$-sheeted covering} (or \textit{$d$-covering} for short) of $G$.

Although $G$ is undirected we shall dually view it as an oriented graph, containing two edges with opposite orientation for each undirected edge. We denote the edges with positive (resp. negative) orientation by $E^{+}(G)$ (resp. $E^{-}(G)$) and identify $E(G)$ with the disjoint union $E^{+}(G) \sqcup E^{-}(G)$. If $e \in E^{\pm}(G)$, then we write $-e$ for the corresponding edge in $E^{\mp}(G)$ with opposite orientation. Moreover we denote by $h(e)$ and $t(e)$, the head and tail of the edge $e \in E(G)$ respectively.
A $d$-covering $H$ of a graph $G$ can be constructed via the following model, introduced in \cite{AL, Fri}. The vertices of $H$ are defined by $V(H) \coloneqq \{ v_i : v \in V(G), 1 \leq i \leq d \}$. The edges of $H$ are determined, as described below, by a labeling $\sigma:E(G) \to S_d$ (see Figure \ref{fig:grplabeling}) satisfying $\sigma(-e) = \sigma(e)^{-1}$. For notational purposes we write $\sigma(e) = \sigma_e$. For every positively oriented edge $e \in E^{+}(G)$ we introduce $d$ (undirected) edges in $H$ connecting $h(e)_i$ to $t(e)_{\sigma_e(i)}$ for $1 \leq i \leq d$, that is, we replace each undirected edge $e$ in $G$ by the perfect matching induced by $\sigma_e$, see Figure \ref{fig:covergraph}. We shall interchangeably refer to the map $\sigma$ and the covering graph $H$ which it determines, as a covering of $G$. Let $\mathcal{C}_{d,G}$ denote the probability space of all $d$-coverings of $G$ endowed with the uniform distribution.

Instead of labeling each edge in $G$ by a permutation in $S_d$ we may label the edges with elements coming from an arbitrary finite group $\Gamma$. A \textit{$\Gamma$-labeling} of a graph $G$ is a function $\gamma:E(G) \to \Gamma$ satisfying $\gamma(-e) = \gamma(e)^{-1}$. Let $\mathcal{C}_{\Gamma, G}$ denote the probability space of all $\Gamma$-labelings of $G$ endowed with the uniform distribution.  Let $\pi:\Gamma \to \text{GL}_d(\mathbb{C})$ be a representation of $\Gamma$. For any $\Gamma$-labeling $\gamma$ of $G$, let $A_{\gamma,\pi}$ denote the $nd \times nd$ matrix obtained from the adjacency matrix $A_G$ of $G$ by replacing the $(i,j)$ entry in $A_G$ with the $d \times d$ block $\sum_{e \in E(G)} \pi(\gamma(e))$ (where the sum runs over all oriented edges from $i$ to $j$) and by a zero block if there are no edges between $i$ and $j$. The matrix $A_{\gamma,\pi}$ is called a \textit{$(\Gamma, \pi)$-covering} of $G$.

\begin{figure}
\begin{tikzpicture}
\coordinate (A) at (0.2,0);
\coordinate (B) at (1.25,-1);
\coordinate (C) at (2.3,0);
\coordinate (D) at (3.9,0);

\draw [very thick,red!80] (A) -- (B);
\draw [very thick,green!80] (A) -- (C);
\draw [very thick,blue!80] (B) -- (C);
\draw [very thick,orange!80] (C) -- (D);
\draw[very thick, purple!80, scale=4] (D.-10)  to[in=45,out=-45,loop] (D.-10);

\draw[fill=black] (A) circle (0.08cm);
\draw[fill=black] (B) circle (0.08cm);
\draw[fill=black] (C) circle (0.08cm);
\draw[fill=black] (D) circle (0.08cm);

\node[yshift = 0.25cm] at (A) {$a$};
\node[yshift = -0.3cm] at (B) {$b$};
\node[yshift = 0.25cm] at (C) {$c$};
\node[xshift=-0.1cm ,yshift = 0.25cm] at (D) {$d$};

\node[xshift = 0.35cm,yshift = -0.7cm] at (A) {$\sigma_1$};
\node[xshift = 1.1cm,yshift = 0.2cm] at (A) {$\sigma_2$};
\node[xshift = 1.8cm,yshift = -0.7cm] at (A) {$\sigma_3$};
\node[xshift = 2.9cm,yshift = 0.2cm] at (A) {$\sigma_4$};
\node[xshift = 5.1cm,yshift = 0cm] at (A) {$\sigma_5$};
\end{tikzpicture}
\caption{A $S_4$-labeling $\gamma$ of a graph $G$ with $\gamma =(\sigma_1,\sigma_2,\sigma_3, \sigma_4,\sigma_5) =  ((1 \thickspace 2),\thickspace (1 \thickspace 2)(3 \thickspace 4),\thickspace (1\thickspace 3 \thickspace 2),\thickspace (1\thickspace 2 \thickspace 3 \thickspace 4),\thickspace (1 \thickspace 2 \thickspace 3))$.}
\label{fig:grplabeling}
\end{figure}
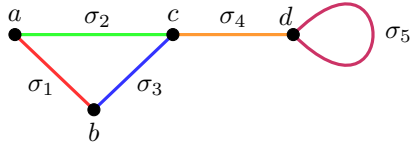
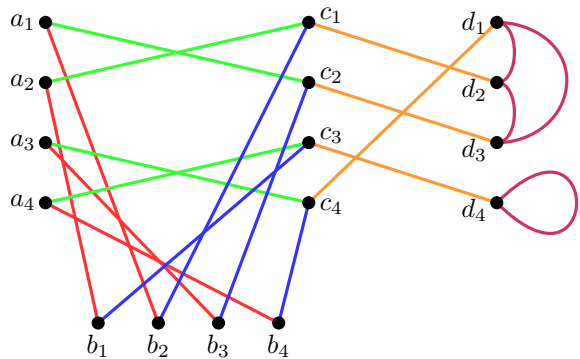
\begin{figure}
\begin{tikzpicture}
\coordinate (A1) at (0,0);
\coordinate (A2) at (0,-0.8);
\coordinate (A3) at (0,-1.6);
\coordinate (A4) at (0,-2.4);


\coordinate (B1) at (0.5+0.2,-4);
\coordinate (B2) at (1.3+0.2,-4);
\coordinate (B3) at (2.1+0.2,-4);
\coordinate (B4) at (2.9+0.2,-4);

\coordinate (C1) at (3.5, 0);
\coordinate (C2) at (3.5,-0.8);
\coordinate (C3) at (3.5,-1.6);
\coordinate (C4) at (3.5,-2.4);

\coordinate (D1) at (6,0);
\coordinate (D2) at (6,-0.8);
\coordinate (D3) at (6,-1.6);
\coordinate (D4) at (6,-2.4);

\draw [very thick,red!80] (A1) -- (B2);
\draw [very thick,red!80] (A2) -- (B1);
\draw [very thick,red!80] (A3) -- (B3);
\draw [very thick,red!80] (A4) -- (B4);

\draw [very thick,green!80] (A1) -- (C2);
\draw [very thick,green!80] (A2) -- (C1);
\draw [very thick,green!80] (A3) -- (C4);
\draw [very thick,green!80] (A4) -- (C3);

\draw [very thick,blue!80] (B1) -- (C3);
\draw [very thick,blue!80] (B2) -- (C1);
\draw [very thick,blue!80] (B3) -- (C2);
\draw [very thick,blue!80] (B4) -- (C4);

\draw [very thick,orange!80] (C1) -- (D2);
\draw [very thick,orange!80] (C2) -- (D3);
\draw [very thick,orange!80] (C3) -- (D4);
\draw [very thick,orange!80] (C4) -- (D1);


\draw[very thick, purple!80, scale=4] (D4.-10)  to[in=45,out=-45,loop] (D4.-10);

\draw[very thick,purple!80,scale=1.1] (D3) to[out=0,in=-90] (6.3,-0.7) to[out=90, in=0] (D1);
\draw[very thick,purple!80] (D3) to[bend right=90, bend angle=90] (D2);
\draw[very thick,purple!80] (D2) to[bend right=90, bend angle=90] (D1);

\draw[fill=black] (A1) circle (0.08cm);
\draw[fill=black] (A2) circle (0.08cm);
\draw[fill=black] (A3) circle (0.08cm);
\draw[fill=black] (A4) circle (0.08cm);

\draw[fill=black] (B1) circle (0.08cm);
\draw[fill=black] (B2) circle (0.08cm);
\draw[fill=black] (B3) circle (0.08cm);
\draw[fill=black] (B4) circle (0.08cm);

\draw[fill=black] (C1) circle (0.08cm);
\draw[fill=black] (C2) circle (0.08cm);
\draw[fill=black] (C3) circle (0.08cm);
\draw[fill=black] (C4) circle (0.08cm);

\draw[fill=black] (D1) circle (0.08cm);
\draw[fill=black] (D2) circle (0.08cm);
\draw[fill=black] (D3) circle (0.08cm);
\draw[fill=black] (D4) circle (0.08cm);

\node[xshift = -0.3cm] at (A1) {$a_1$};
\node[xshift = -0.3cm] at (A2) {$a_2$};
\node[xshift = -0.3cm] at (A3) {$a_3$};
\node[xshift = -0.3cm] at (A4) {$a_4$};

\node[yshift = -0.3cm] at (B1) {$b_1$};
\node[yshift = -0.3cm] at (B2) {$b_2$};
\node[yshift = -0.3cm] at (B3) {$b_3$};
\node[yshift = -0.3cm] at (B4) {$b_4$};

\node[xshift = 0.3cm, yshift = 0.1cm] at (C1) {$c_1$};
\node[xshift = 0.3cm, yshift = 0.1cm] at (C2) {$c_2$};
\node[xshift = 0.3cm, yshift = 0.1cm] at (C3) {$c_3$};
\node[xshift = 0.3cm, yshift = -0.05cm] at (C4) {$c_4$};

\node[xshift = -0.3cm] at (D1) {$d_1$};
\node[xshift = -0.3cm, yshift = -0.1cm] at (D2) {$d_2$};
\node[xshift = -0.3cm, yshift = -0.1cm] at (D3) {$d_3$};
\node[xshift = -0.3cm, yshift = -0.1cm] at (D4) {$d_4$};

\end{tikzpicture}
\caption{The $4$-sheeted covering graph $H$ corresponding to the $S_4$-labeling $\gamma$ of $G$ in Figure \ref{fig:grplabeling}.} 
\label{fig:covergraph}
\end{figure}\noindent

Consider the $d$-dimensional representation $\pi:S_d \to \text{GL}_d(\mathbb{C})$ of the symmetric group $S_d$ mapping every $\sigma \in S_d$ to its corresponding permutation matrix. The representation $\pi$ is reducible since the $1$-dimensional space $\langle \mathbf{1} \rangle \leq \mathbb{C}^d$, where $\mathbf{1} = (1,\dots, 1)$, is invariant under the action of $\pi$. The action of $\pi$ on the orthogonal complement $\langle \mathbf{1} \rangle^{\perp}$ is an irreducible $(d-1)$-dimensional representation called the \textit{standard representation}, denoted $\text{std}:S_d \to \text{GL}_{d-1}(\mathbb{C})$. As outlined in \cite{HPS}, every $d$-covering $H$ of $G$ corresponds uniquely to a $(S_d, \text{std})$-covering of $G$.
\subsection{Stable polynomials}
A polynomial $f(\mathbf{x}) \in \mathbb{C}[x_1,\dots, x_n]$ is said to be \textit{stable} if $f(x_1,\dots, x_n) \neq 0$ whenever $\text{Im}(x_j) > 0$ for all $j = 1,\dots, n$. By convention we also regard the zero polynomial to be stable. A stable polynomial with only real coefficients is said to be \textit{real stable}. Note that univariate real stable polynomials are precisely the real-rooted polynomials (i.e. real polynomials in one variable with all zeros in $\mathbb{R}$). Thus stability may be regarded as a multivariate generalization of real-rootedness. Below we collect a few facts about stable polynomials which are relevant for the forthcoming sections. For a more comprehensive background we refer to the survey by Wagner \cite{Wag} and references therein. 

A common technique for proving that a polynomial $f(\mathbf{x})$ is stable is to realize $f(\mathbf{x})$ as the image of a known stable polynomial under a stability preserving linear transformation.
Stable polynomials satisfy several basic closure properties, among them are diagonalization $f \mapsto f(\mathbf{x})|_{x_i = x_j}$ for $i,j \in [n]$ and differentiation $f \mapsto \partial_i f(\mathbf{x})$ where $\partial_i \coloneqq \frac{\partial}{\partial x_i}$. The following theorem by Lieb and Sokal provides the construction for a large family of stability preserving linear transformations.
\begin{theorem}[Lieb-Sokal \cite{LS}] \label{thm:liebsokal}
If $f(x_1,\dots, x_n) \in \mathbb{C}[x_1,\dots, x_n]$ is a stable polynomial, then $f\left( \partial_1,\dots, \partial_n \right )$ is a stability preserving linear operator. 
\end{theorem} \noindent
Borcea and Br\"and\'en \cite{BBa} gave a complete characterization of the linear operators preserving stability. The following is the transcendental characterization of stability preservers on infinite-dimensional complex polynomial spaces. Define the \textit{complex Laguerre-P\'olya class} to be the class of entire functions in $n$ variables that are limits, uniformly on compact sets of stable polynomials in $n$ variables. 
Throughout we will use the following multi-index notation 
\[
\mathbf{x}^{S} \coloneqq \prod_{i\in S}^n x_i \hspace{1cm} \mathbf{x}^{\boldsymbol{\alpha}} \coloneqq \prod_{i=1}^n x_i^{\alpha_i}, \hspace{1cm} \boldsymbol{\alpha}! \coloneqq \prod_{i=1}^n \alpha_i!,
\] 
where $S\subseteq [n]$ and $\boldsymbol{\alpha} = (\alpha_i) \in \mathbb{N}^n$.
\begin{theorem}[Borcea-Br\"and\'en \cite{BBa}] \label{thm:borbra}
Let $T: \mathbb{C}[x_1,\dots, x_n] \to \mathbb{C}[x_1,\dots, x_n]$ be a linear operator. Then $T$ preserves stability if and only if either
\begin{enumerate}
\item $T$ has range of dimension at most one and is of the form
\[
T(f) = \alpha(f) P,
\]
where $\alpha$ is a linear functional on $\mathbb{C}[x_1,\dots, x_n]$ and $P$ is a stable polynomial, or
\item 
\[
G_T(\mathbf{x}, \mathbf{y}) \coloneqq \sum_{\alpha \in \mathbb{N}^n} (-1)^{\alpha} T(\mathbf{x}^{\boldsymbol{\alpha}}) \frac{\mathbf{y}^{\boldsymbol{\alpha}}}{\boldsymbol{\alpha}!}
\]
belongs to the Laguerre-P\'olya class.
\end{enumerate}
\end{theorem} \noindent
A polynomial $f(x_1,\dots, x_n)$ is said to be \textit{multiaffine} if each variable $x_i$ occurs with degree at most one in $f(x_1,\dots, x_n)$, and is called \textit{symmetric} if $f(x_{\sigma(1)},\dots, x_{\sigma(n)}) = f(x_1,\dots,x_n)$ for all $\sigma \in S_n$. The Grace-Walsh-Szeg\"o coincidence theorem is a cornerstone in the theory of stable polynomials frequently used to depolarize symmetries before checking stability.  One version of it is stated below, see \cite{BBa, Wag} for modern references and alternative proofs.
\begin{theorem}[Grace-Walsh-Szeg\"o \cite{Gra,Sze, Wal}] \label{thm:gws}
Let $f(x_1,\dots, x_n) \in \mathbb{C}[x_1,\dots, x_n]$ be a symmetric and multiaffine polynomial. Then $f(x_1,\dots, x_n)$ is stable if and only if $f(x,\dots, x)$ is stable.
\end{theorem} \noindent

\section{Stability of multivariate $d$-matching polynomials}
A \textit{matching} of an undirected graph $G$ is a subset $M \subseteq E(G)$ such that no two edges in $M$ share a common vertex. Let $V(M) \coloneqq \bigcup_{\{i,j\} \in M} \{i,j \}$ denote the set of vertices in the matching $M$.
For $d \in \mathbb{Z}_{\geq 1}$, the \textit{$d$-matching polynomial} of $G$ is defined by
\[
\mu_{d,G}(x) \coloneqq \mathbb{E}_{H \in \mathcal{C}_{d,G}} \mu_H(x),
\]
where 
\[
\mu_G(x) \coloneqq \sum_{i=0}^{\lfloor n/2 \rfloor} (-1)^i m_i x^{n-2i} \in \mathbb{Z}[x]
\]
and $m_i$ denotes the number of matchings in $G$ of size $i$ with $m_0=1$. In particular if $d=1$, then $\mu_{d,G}(x)$ coincides with the conventional matching polynomial $\mu_G(x)$. The following results are proved in \cite{HPS}.

\begin{theorem}[Hall-Puder-Sawin \cite{HPS}] \label{thm:expdmatch}
Let $\Gamma$ be a finite group and $\pi:\Gamma \to \text{GL}_d(\mathbb{C})$ be an irreducible representation such that $\pi(\Gamma)$ is a complex reflection group, i.e., $\pi(\Gamma)$ is generated by pseudo-reflections. If $G$ is a finite connected multigraph, then
\begin{align} \label{thm:hpsavg}
\mathbb{E}_{\gamma \in \mathcal{C}_{\Gamma,G}} \det(xI - A_{\gamma, \pi}) =\mu_{d,G}(x).
\end{align}
\end{theorem} \noindent
\begin{remark}
Remarkably the expected characteristic polynomial in \eqref{thm:hpsavg} depends only on the dimension $d$ of the irreducible representation $\pi$ and not on the particular choice of group $\Gamma$, nor the specifics of the map $\pi$. Real-rooted expected characteristic polynomials have seen a surge of interest recently in light of the Kadison-Singer problem and Ramanujan coverings, see e.g. \cite{AG,HPS,MSS,MSSc, MSSb,MSSd}.
\end{remark} \noindent
\begin{example}
A classical result due to Godsil and Gutman \cite{GG} states that if $A = (a_{ij})$ is the adjacency matrix of a finite simple undirected graph $G$, then
\[
\mathbb{E}_{\mathbf{s}} \det(xI - A^{\mathbf{s}}) = \mu_G(x),
\]
where $A^{\mathbf{s}}_{ij} \coloneqq s_{e}a_{ij}$ for all $e = \{i,j\} \in E(G)$ and  $\mathbf{s} = (s_{e})_{e \in E(G)} \in \{\pm 1 \}^{E(G)}$. In other words, the expected characteristic polynomial over all signings of $G$ equals the matching polynomial of $G$.
In the language of Hall, Puder and Sawin this corresponds to taking $\Gamma = \mathbb{Z}/2\mathbb{Z}$ and $\pi = \text{sgn}:\mathbb{Z}/2\mathbb{Z} \to \text{GL}_1(\mathbb{C})$ to be the sign representation in Theorem \ref{thm:hpsavg}.  
\end{example}
Generalizing and extending the following theorem will be the main focus of this section.
\begin{theorem}[Hall-Puder-Sawin \cite{HPS}] \label{thm:dmatchrealroots}
If $G$ is a finite multigraph with no loops, then $\mu_{d,G}(x)$ is real-rooted.
\end{theorem} \noindent
\begin{remark}
Hall, Puder and Sawin also showed that the roots of $\mu_{d,G}(x)$ are contained inside the Ramanujan interval of $G$ (see \cite{HPS}).
\end{remark} \noindent
Define the \textit{multivariate $d$-matching polynomial} of $G$ by
\[
\mu_{d,G}(\mathbf{x}) \coloneqq \mathbb{E}_{H \in \mathcal{C}_{d,G}} \mu_H(\mathbf{x}),
\]
where
\[
\mu_G(\mathbf{x}) \coloneqq \sum_{M} (-1)^{|M|} \prod_{i \in [n]\setminus V(M)} x_i,
\]
and the sum runs over all matchings in $G$. 

The real-rootedness of $\mu_{d,G}(x)$ was proved indirectly in \cite{HPS} by considering a limit of interlacing families converging to the left-hand side in Theorem \ref{thm:expdmatch}. In this section we use a more direct approach for proving the real-rootedness of $\mu_{d,G}(x)$. In fact we prove something stronger, namely that $\mu_{d,G}(\mathbf{x})$ is stable. Our proof also holds for graphs with loop edges, thus removing the redundant hypothesis in Theorem \ref{thm:dmatchrealroots}. 

Coverings of graphs with loop edges have interesting properties.  
In particular, consider the $|\Gamma|$-dimensional \textit{regular representation} $\text{reg}: \Gamma \to \text{GL}_{|\Gamma|}(\mathbb{C})$ sending an element $g \in \Gamma$ to the permutation matrix afforded by $g$ acting on $\Gamma$ through left translation $h \mapsto gh$. The \textit{bouquet graph} $B_r$ is the graph consisting of a single vertex with $r$ loop edges. A $(\Gamma,\text{reg})$-covering $A_{\gamma, \text{reg}}$ of $B_r$ is equivalent to the Cayley graph of $\Gamma$ with respect to the set $\gamma(E(B_r))$. In this sense $(\Gamma,\text{reg})$-coverings of finite multigraphs with loops generalize Cayley graphs of finite groups.
\begin{figure}
\begin{tikzpicture}[scale=5]
\draw[very thick, blue!80] (0,0)  to[in=40+35,out=-40+35,loop] (0,0);
\draw[very thick, red!80] (0,0)  to[in=40+145,out=-40+145,loop] (0,0);
\fill (0,0) circle[radius=0.6pt];
\node[xshift=-1.45cm, yshift=0.9cm] at (0,0) {$\sigma_1$};
\node[xshift=1.45cm,yshift=0.9cm] at (0,0) {$\sigma_2$};
\end{tikzpicture}
\begin{tikzpicture}[scale=1.1]
\coordinate (A) at (0+0.8, -0.1);
\coordinate (B) at (1.6+0.8, -0.3);
\coordinate (C) at (1+0.8, 1.2);
\coordinate (D) at (0+3, -0.1+1);
\coordinate (E) at (1.6+3, -0.3+1);
\coordinate (F) at (1+3, 1.2+1);



\begin{scope}[very thick, blue!80, every node/.style={sloped,allow upside down}]
\draw[very thick,blue!80] (A) to[bend left=10, bend angle=45] node {\midarrow} (D);
\draw[very thick,blue!80] (B) to[bend left=10, bend angle=45] node {\midarrow} (E);
\draw[very thick,blue!80] (C) to[bend left=10, bend angle=45] node {\midarrow} (F);
\end{scope}

\begin{scope}[very thick, blue!80, every node/.style={sloped,allow upside down}]
\draw[very thick,blue!80] (D) to[bend left=10, bend angle=45] node {\midarrow} (A);
\draw[very thick,blue!80] (E) to[bend left=10, bend angle=45] node {\midarrow} (B);
\draw[very thick,blue!80] (F) to[bend left=10, bend angle=45] node {\midarrow} (C);
\end{scope}

\begin{scope}[very thick, red!80, every node/.style={sloped,allow upside down}]
  \draw (A)-- node {\midarrow} (C);
  \draw (C)-- node {\midarrow} (B);
  \draw (B)-- node {\midarrow} (A);
\end{scope}

\begin{scope}[very thick, red!80, every node/.style={sloped,allow upside down}]
  \draw (D)-- node {\midarrow}(E);
  \draw (E)-- node {\midarrow} (F);
  \draw (F)-- node {\midarrow} (D);
\end{scope}

\draw[fill=black] (A) circle (0.08cm);
\draw[fill=black] (B) circle (0.08cm);
\draw[fill=black] (C) circle (0.08cm);
\draw[fill=black] (D) circle (0.08cm);
\draw[fill=black] (E) circle (0.08cm);
\draw[fill=black] (F) circle (0.08cm);

\node[xshift=-0.2cm] at (A) {$\iota$};
\node[xshift=-0.6cm] at (C) {$(1 \thinspace 2 \thinspace 3)$};
\node[xshift=0.5cm, yshift=-0.3cm] at (B) {$(1 \thinspace 3 \thinspace 2)$};

\node[xshift=-0.4cm, yshift=0.2cm] at (D) {$(1 \thinspace 2)$};
\node[xshift=0.5cm] at (E) {$(1 \thinspace 3)$};
\node[xshift=-0.4cm, yshift = 0.2cm] at (F) {$(2 \thinspace 3)$};
\end{tikzpicture}
\caption{A $S_3$-labeling $\gamma = (\sigma_1,\sigma_2) = ((1 \thinspace 2 \thinspace 3), (1 \thinspace 2))$ of the bouquet graph $B_2$ (left) and the Cayley graph of $S_3$ with respect to $\{(1 \thinspace 2 \thinspace 3), (1 \thinspace 2) \}$ (right). }
\label{fig:cayley}
\end{figure}
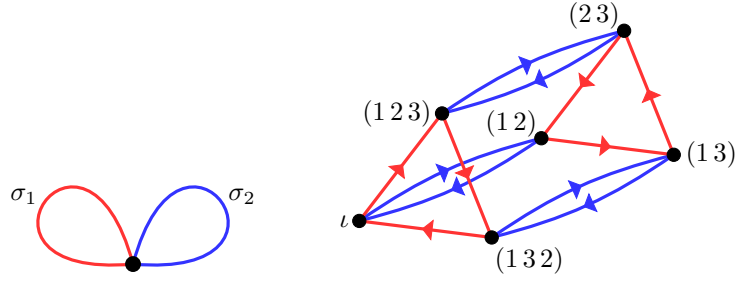
\begin{example} 
Let $\Gamma = S_3$, $G = B_2$ and $\pi = \text{reg}:S_3 \to \text{GL}_6(\mathbb{C})$. Consider the $S_3$-labeling $\gamma = (\sigma_1,\sigma_2) = ((1\thinspace 2 \thinspace 3), (1 \thinspace 2))$ of $B_2$ as in Figure \ref{fig:cayley} (left). Then
\[
A \coloneqq \text{reg}(\sigma_1)+  \text{reg}(\sigma_2) = \bordermatrix{
    & \iota & (1 \thinspace 2) & (1 \thinspace 3) & (2 \thinspace 3) & (1 \thinspace 2 \thinspace 3) & (1 \thinspace 3 \thinspace 2) \cr
    \iota            			  & 0 & \color{blue}1 & 0 & 0 & \color{red}1 & 0 \cr
    (1 \thinspace 2) 			  & \color{blue}1 & 0 & \color{red}1 & 0 & 0 & 0 \cr
    (1 \thinspace 3)			  & 0 & 0 & 0 & \color{red}1 & 0 & \color{blue}1 \cr
    (2 \thinspace 3) 		      & 0 & \color{red}1 & 0 & 0 & \color{blue}1 & 0 \cr
    (1 \thinspace 2 \thinspace 3) & 0 & 0 & 0 & \color{blue}1 & 0 & \color{red}1 \cr
    (1 \thinspace 3 \thinspace 2) & \color{red}1 & 0 & \color{blue}1 & 0 & 0 & 0
  },
\]
is the adjacency matrix of the Cayley graph of $S_3$ with respect to the set $\{\sigma_1,\sigma_2 \}$, see Figure \ref{fig:cayley} (right), and the $(S_3,\text{reg})$-covering $A_{\gamma, \text{reg}}$ is given by $A + A^{T}$.
\end{example} \noindent
Choe, Oxley, Sokal and Wagner \cite{COSW} (see also \cite{BBb}) consider the \textit{multi-affine part} operator
\begin{align*}
\text{MAP}: \mathbb{C}[x_1,\dots, x_n] &\to \mathbb{C}[x_1,\dots, x_n] \\
\sum_{\boldsymbol{\alpha} \in \mathbb{N}^n} a(\boldsymbol{\alpha}) \mathbf{x}^{\boldsymbol{\alpha}} &\mapsto \sum_{\boldsymbol{\alpha} : \alpha_i \leq 1, i \in [n]} a(\boldsymbol{\alpha}) \mathbf{x}^{\boldsymbol{\alpha}}
\end{align*} \noindent
and note that it is a stability preserving linear operator. Indeed the symbol
\[
G_{\text{MAP}}(\mathbf{x}, \mathbf{y}) = \sum_{\boldsymbol{\alpha} \in \mathbb{N}^n} (-1)^{\boldsymbol{\alpha}} \text{MAP}(\mathbf{x}^{\boldsymbol{\alpha}}) \frac{\mathbf{y}^{\boldsymbol{\alpha}}}{\boldsymbol{\alpha}!}  = \sum_{\substack{ \boldsymbol{\alpha} \in \mathbb{N}^n\\ \alpha_i \leq 1 }} (-1)^{\boldsymbol{\alpha}} \mathbf{x}^{\boldsymbol{\alpha}} \mathbf{y}^{\boldsymbol{\alpha}} = \prod_{i=1}^n (1-x_iy_i),
\]
is stable being a product of stable polynomials. Since the range of $\text{MAP}$ has dimension greater than one, it follows that $\text{MAP}$ preserves stability by Theorem \ref{thm:borbra}. Given the identity
\begin{align} \label{eq:subgraphgen}
P_G(\mathbf{x}) \coloneqq \sum_{E \subseteq E(G)} (-1)^{|E|} \prod_{i=1}^n x_i^{\text{deg}_{G[E]}(i)} = \prod_{\{i,j\} \in E(G)} (1- x_ix_j),
\end{align} \noindent
where $G[E]$ is the subgraph of $G$ induced by $E \subseteq E(G)$ and $\text{deg}_{G[E]}(i)$ denotes the degree of $i$ in $G[E]$, we have that

\[
\text{MAP} [P_G(\mathbf{x})] = \mu_G(\mathbf{x}),
\]
and hence that $\mu_G(\mathbf{x})$ is stable being the image of a stable polynomial under $\text{MAP}$. This result is also known as the Heilmann-Lieb theorem \cite{HL}.

By using Theorem \ref{thm:gws} and the stability preserving linear operator $\text{MAP}$ we will show below that $\mu_{d,G}(\mathbf{x})$ is stable.
\begin{theorem} \label{thm:multidmatchpoly}
If $G$ is a finite multigraph (possibly with loops), then $\mu_{d,G}(\mathbf{x})$ is stable for all $d \geq 1$.
\end{theorem}
\begin{proof}
For a $d$-covering $\sigma:E(G) \to S_d$, let 
\begin{align*} 
P_{\sigma,G}(\mathbf{x}) &\coloneqq \left (\prod_{e \in E^+(G) \setminus E_{\circ}^{+}(G)} \prod_{k=1}^d (1 - x_{h(e)k} x_{t(e) \sigma_e(k)}) \right ) \times \\ & \hspace{0.5cm} \left ( \prod_{e \in E_{\circ}^{+}(G)} \prod_{k : \sigma_e(k) \neq k} (1 - x_{h(e)k} x_{t(e) \sigma_e(k)}) \right ),
\end{align*} \noindent
where $E_{\circ}^{+}(G) \coloneqq \{ e \in E^{+}(G) : h(e) = t(e) \}$ denotes the set of positively oriented loops in $G$.
Since no matching may contain loops we have excluded the factors $(1-x_{ik}^2)$ from the subgraph generating polynomial $P_H(\mathbf{x})$ in (\ref{eq:subgraphgen}) where $H$ is the covering graph corresponding to $\sigma$. This explains the form of $P_{\sigma, G}(\mathbf{x})$. It follows that
\[
\text{MAP}[P_{\sigma, G}(\mathbf{x})] = \mu_H(\mathbf{x}).
\]
We have
\begin{align*}
\mathbb{E}_{\sigma \in \mathcal{C}_{d, G}} P_{\sigma, G}(\mathbf{x}) &= \sum_{\sigma \in \mathcal{C}_{d, G}} \frac{1}{|\mathcal{C}_{d, G}|} P_{\sigma, G}(\mathbf{x}) \\ &= \frac{1}{|\mathcal{C}_{d, G}|} \left (\prod_{e \in E^{+}(G) \setminus E_{\circ}^{+}(G)} \sum_{\sigma_e \in S_d} \prod_{k=1}^d (1 - x_{h(e)k} x_{t(e) \sigma_e(k)}) \right ) \times \\ & \hspace*{1.4cm}\left ( \prod_{e \in E_{\circ}^{+}(G)} \sum_{\sigma_e \in S_d} \prod_{k : \sigma_e(k) \neq k} (1 - x_{h(e)k} x_{t(e) \sigma_e(k)}) \right ). 
\end{align*} \noindent
For $e\in E^{+}(G) \setminus E_{\circ}^{+}(G)$ the polynomials 
\[
f_e(\mathbf{x}) = \sum_{\sigma_e \in S_d} \prod_{k=1}^d (1-x_{h(e)k}x_{t(e)\sigma_e(k)}),
\] 
are symmetric and multiaffine polynomials in the two sets of variables 
\[
\{ x_{h(e)k} : 1 \leq k \leq d \} \hspace{0.5cm} \text{and} \hspace{0.5cm} \{ x_{t(e)k} : 1 \leq k \leq d \},
\] 
respectively. By Theorem \ref{thm:gws} we have that $f_e(\mathbf{x})$ is stable if and only if 
\[
\sum_{\sigma_e \in S_d} \prod_{k=1}^d (1-xy) = d!(1-xy)^d,
\]
is stable, the latter of which is clear. Similary if $e\in E_{\circ}^{+}(G)$, then $f_e(\mathbf{x})$ is symmetric and multiaffine in the set of variables $\{x_{h(e)k} : 1\leq k\leq d \}$, so checking stability of $f_e(\mathbf{x})$ reduces by Theorem \ref{thm:gws} to checking the stability of $d!(1-x^2)^d$, which is again clear.
Hence $\mathbb{E}_{\sigma \in \mathcal{C}_{S_d, G}} P_{\sigma, G}(\mathbf{x})$ is stable being a product of stable polynomials. Finally we have
\begin{align*}
\text{MAP}\left [ \mathbb{E}_{\sigma \in \mathcal{C}_{d, G}} P_{\sigma, G}(\mathbf{x}) \right ] &= \mathbb{E}_{\sigma \in \mathcal{C}_{d, G}} \text{MAP} [P_{\sigma, G}(\mathbf{x})] \\ &= \mathbb{E}_{H \in \mathcal{C}_{d, G}} \mu_H(\mathbf{x}) \\ &= \mu_{d,G}(\mathbf{x}).
\end{align*} \noindent
Hence $\mu_{d,G}(\mathbf{x})$ is stable.
\end{proof}
\begin{corollary}
If $G$ is a finite multigraph (possibly with loops), then $\mu_{d,G}(x)$ is real-rooted for all $d \geq 1$.
\end{corollary}
\begin{proof}
Follows by putting $\mathbf{x} = (x,\dots, x)$ in Theorem \ref{thm:multidmatchpoly}
\end{proof} \noindent
\section{Stable expected matching polynomials over induced subgraphs}
In the previous section we considered stable averages of multivariate matching polynomials over the set of $d$-sheeted covering graphs of $G$. In this section we consider stable averages over (vertex-) induced subgraphs of $G$. To this end, if $S \subseteq [n]$, let $G[S]$ denote the subgraph of $G$ induced by the vertices in $S$. 
Let $\mathbb{P}$ be a probability distribution on the power set $\mathcal{P}([n]) \coloneqq \{S : S \subseteq [n] \}$. The polynomial
\begin{align*}
Z_{\mathbb{P}}(\mathbf{x}) = \sum_{S \subseteq [n]} \mathbb{P}(S) \mathbf{x}^S \in \mathbb{R}[x_1,\dots, x_n],
\end{align*}
is called the \textit{partition function} of $\mathbb{P}$. The probability distribution $\mathbb{P}$ is called \textit{Rayleigh} if 
\begin{align}\label{eq:rayleigh}
Z_{\mathbb{P}}(\mathbf{x}) \frac{\partial^2 Z_{\mathbb{P}}(\mathbf{x})}{\partial x_i \partial x_j} \leq \frac{\partial Z_{\mathbb{P}}(\mathbf{x})}{\partial x_i} \frac{\partial Z_{\mathbb{P}}(\mathbf{x})}{\partial x_j}
\end{align} \noindent
for all $\mathbf{x} \in \mathbb{R}_{+}^n$, $1 \leq i,j \leq n$ and is called \textit{strong Rayleigh} if \eqref{eq:rayleigh} holds for all $\mathbf{x}\in \mathbb{R}^n$, $1 \leq i,j \leq n$.
\begin{theorem}[Br\"and\'en \cite{Bra}] \label{thm:brandenrayleigh}
A probability distribution $\mathbb{P}$ is strong Rayleigh if and only if $Z_{\mathbb{P}}(\mathbf{x})$ is stable.
\end{theorem}
\begin{proposition} \label{prop:rayleighstable}
Let $G = (V(G), E(G))$ be a finite undirected graph on $[n]$ and let $\mathbb{P}$ be a probability distribution on $\mathcal{P}([n])$. If $\mathbb{P}$ is strong Rayleigh, then
$\mathbb{E}_{S \subseteq [n]}^{\mathbb{P}} \mu_{G[S]}(\mathbf{x})$ is stable.
\end{proposition}
\begin{proof}
By Theorem \ref{thm:liebsokal} the linear operator 
\[
T_G \coloneqq \prod_{\{i,j\} \in E(G)} (1 - \partial_i\partial_j)
\]
preserves stability. Moreover it is easy to see that for $S \subseteq [n]$,
\[
T_G(\mathbf{x}^S) = \mu_{G[S]}(\mathbf{x}).
\]
If $\mathbb{P}$ is strong Rayleigh, then $Z_{\mathbb{P}}(\mathbf{x})$ is stable by Theorem \ref{thm:brandenrayleigh}. Hence
\begin{align*}
T_G(Z_{\mathbb{P}}(\mathbf{x})) &= \sum_{S \subseteq [n]} \mathbb{P}(S) T_G(\mathbf{x}^S) \\ &= \sum_{S \subseteq [n]} \mathbb{P}(S) \mu_{G[S]}(\mathbf{x}) \\ &= \mathbb{E}_{S \in \subseteq [n]}^{\mathbb{P}} \mu_{G[S]}(\mathbf{x}),
\end{align*} \noindent
is stable.
\end{proof}
\begin{corollary} \label{cor:avgreal}
If $\mathbb{P}$ is a strong Rayleigh probability distribution, then $\mathbb{E}_{S \subseteq [n]}^{\mathbb{P}} \mu_{G[S]}(x)$ is real-rooted.
\end{corollary}

\begin{example}
The following example demonstrates that the converse to Proposition \ref{prop:rayleighstable} is false. Consider the graph $G = \bullet\textendash \bullet$ on two vertices and one edge. If $\mathbb{P}$ is a probability distribution with $\mathbb{P}(\{1,2\}) = a$, $\mathbb{P}(\{1\}) = b$, $\mathbb{P}(\{1,2\}) = c$ and $\mathbb{P}(\emptyset) = d$, then
\begin{align*}
\mathbb{E}_{S \subseteq [2]}^{\mathbb{P}} \mu_{G[S]}(\mathbf{x}) &= a \mu_G(\mathbf{x}) + b\mu_{G[1]}(\mathbf{x}) + c\mu_{G[2]}(\mathbf{x}) + d\mu_{G[\emptyset]}(\mathbf{x}) \\ &= a(-1+x_1x_2) + bx_1 + cx_2 + d,
\end{align*} \noindent
which is stable if and only of $bc - a(-a+d) \geq 0$. On the other hand
\[
Z_{\mathbb{P}}(\mathbf{x}) = ax_1x_2 + bx_1 + cx_2+d,
\]
is stable if and only if $bc-ad\geq 0$. Hence there exists probability distributions $\mathbb{P}$ which are not strong Rayleigh for which $\mathbb{E}_{S \subseteq [n]}^{\mathbb{P}} \mu_{G[S]}(\mathbf{x})$ is stable. An interesting question would be to characterize the probability distributions for which $\mathbb{E}_{S \subseteq [n]}^{\mathbb{P}} \mu_{G[S]}(\mathbf{x})$ is stable.
\end{example}
\begin{example}
A natural probability distribution $\mathbb{P}$ on the set of induced subgraphs of $G$ is the Bernoulli distribution $\mathbb{B}$ where a vertex $i \in [n]$ is selected independently with probability $p_i$ and not selected with probability $1-p_i$. Note that $\mathbb{B}$ is a strong Rayleigh probability distribution since 
\[
Z_{\mathbb{B}}(\mathbf{x}) = \sum_{S \subseteq [n]} \mathbb{B}(S) \mathbf{x}^S = \sum_{S \subseteq [n]} \prod_{i \in S}p_i \prod_{i \in [n] \setminus S} (1-p_i) \mathbf{x}^S = \prod_{i=1}^n ((1-p_i)+  p_ix_i),
\]
is stable.
Hence $\mathbb{E}_{S \subseteq [n]}^{\mathbb{B}} \mu_{G[S]}(\mathbf{x})$ is stable by Proposition \ref{prop:rayleighstable}.
\end{example} \noindent
Next we shall provide bounds for the real roots of $\mathbb{E}_{S \subseteq [n]}^{\mathbb{P}} \mu_{G[S]}(x)$. 

Let $i \in V(G)$ and define a graph $U_i(G)$ with vertex set being the set of all non-backtracking walks in $G$ starting from $i$, i.e., sequences $(i_0,i_1,\dots, i_k)$ such that $i_0 = i$, $i_r$ and $i_{r+1}$ are adjacent and $i_{r+1} \neq i_{r-1}$. Two such walks are connected by an edge in $U_i(G)$ if one walk extends the other by one vertex, i.e., $(i_0,\dots, i_k, i_{k+1})$ is adjacent to $(i_0,\dots, i_k)$. The graph thus constructed is a tree that covers $G$. It is called the \textit{universal covering tree} of $G$. The universal covering tree $U_i(G)$ of $G$ is unique up to isomorphism and has the property that it covers every other covering of $G$. Thus we henceforth remove reference to the root $i$ and write $U(G)$ for the universal covering tree of $G$. The tree $U(G)$ is countably infinite, unless $G$ is a finite tree, in which case $U(G) = G$. 

The \textit{spectral radius} $r(G)$ of a finite graph $G$ is the largest absolute eigenvalue of the adjacency matrix $A_G$ of $G$.
By a theorem of Mohar \cite{Moh} the spectral radius of an infinite graph $U$ can be defined as follows,
\[
r(U) \coloneqq \text{sup}\{ r(G) : \text{$G$ is a finite induced subgraph of $U$}  \}.
\]
If $G$ is a finite undirected graph, then let $\rho(G) \coloneqq r(U(G))$ denote the spectral radius of its universal covering tree. Say that a probability distribution $\mathbb{P}$ on $\mathcal{P}([n])$ has \textit{constant parity} if the set $\{ |S| : S\subseteq [n], \thickspace \mathbb{P}(S) > 0 \}$ consists of numbers with the same parity (i.e. are either all odd or all even).
\begin{proposition} \label{prop:univspanbnd}
Let $G$ be a finite undirected graph with $n$ vertices and $\mathbb{P}$ a probability distribution on $\mathcal{P}([n])$. Then the real roots of $\mathbb{E}_{S \subseteq [n]}^{\mathbb{P}} \mu_{G[S]}(x)$ are bounded above by $\rho(G)$. Moreover if $\mathbb{P}$ has constant parity, then the real roots of $\mathbb{E}_{S \subseteq [n]}^{\mathbb{P}} \mu_{G[S]}(x)$ are contained in $[-\rho(G), \rho(G)]$.
\end{proposition}
\begin{proof}
Let $S \subseteq [n]$. There is a clear injective embedding of $U(G[S])$ into $U(G)$ such that any finite induced subgraph of $U(G[S])$ is an induced subgraph of $U(G)$. Therefore $\rho(G[S])\leq \rho(G)$. Heilmann and Lieb \cite{HL} showed that for any finite graph $G$, the roots of $\mu_G(x)$ are contained in $[-\rho(G), \rho(G)]$. Therefore $\mathbb{E}_{S \subseteq [n]}^{\mathbb{P}} \mu_{G[S]}(x) > 0$ for all $x \in (\rho(G), \infty)$, being a convex combination of monic polynomials with the same property. Hence the real roots of the expectation are bounded above by $\rho(G)$. If $\mathbb{P}$ also has constant parity, then $\mathbb{E}_{S \subseteq [n]}^{\mathbb{P}} \mu_{G[S]}(x)$ is a convex combination of monic polynomials with same degree parity and are therefore, by above, strictly positive or strictly negative on the interval $(-\infty, -\rho(G))$. Hence the real roots are contained in $[-\rho(G), \rho(G)]$.
\end{proof}
\begin{corollary}
Let $G$ be a finite undirected graph on $n$ vertices and $k \in [n]$. Then the uniform average of all matching polynomials over the set of induced size $k$-subgraphs of $G$ is a real-rooted polynomial with all roots contained in the interval $[-\rho(G), \rho(G)]$.
\end{corollary}
\begin{proof}
Let $\mathbb{P}$ be the probability distribution on $\mathcal{P}([n])$ with uniform support on $\binom{[n]}{k}$. Then
\[
Z_{\mathbb{P}}(\mathbf{x}) = \frac{1}{\binom{n}{k}} e_k(\mathbf{x}),
\]
where $e_k(\mathbf{x})$ denotes the elementary symmetric polynomial of degree $k$. The polynomial $e_k(\mathbf{x})$ is stable, e.g. by Theorem \ref{thm:gws}. Therefore $\mathbb{P}$ is a strong Rayleigh probability distribution by Theorem \ref{thm:brandenrayleigh}, so the statement follows by Corollary \ref{cor:avgreal} and Proposition \ref{prop:univspanbnd}.
\end{proof}
\section{Stable relaxed matching polynomials}
A hypergraph $H= (V(H), E(H))$ is a set of vertices $V(H) = [n]$ together with a family of subsets $E(H)$ of $V(H)$ called \textit{hyperedges} (or \textit{edges} for short). The \textit{degree} of a vertex $i \in V(H)$ is defined as $\text{deg}_H(i) \coloneqq |\{e \in E(H) : i \in e \}|$.
In analogy with graph matchings, a matching in a hypergraph consists of a subset of edges with empty pairwise intersection.
Although the matching polynomial $\mu_G(x)$ of a graph $G$ is real-rooted, the analogous polynomial for hypergraphs is not real-rooted in general, see e.g. \cite{GZM}. From the point of view of real-rootedness we consider a weaker notion of matchings that provide a natural generalization of the real-rootedness property of $\mu_G(x)$ to hypergraphs. 

Let $H = (V(H),E(H))$ be a hypergraph.
Define a \textit{relaxed matching} in $H$ to be a collection $M = (S_e)_{e \in E}$ of edge subsets such that $E \subseteq E(H)$, $S_e \subseteq e$,  $|S_e| >  1$ and $S_e \cap S_{e'} = \emptyset$ for all pairwise distinct $e,e' \in E$ (see Figure \ref{fig:relaxmatch}). 
\begin{remark}
If $H$ is a graph then the concept of relaxed matching coincides with the conventional notion of graph matching. Note also that a conventional hypergraph matching is a relaxed matching $M = (S_e)_{e \in E}$ for which $S_e = e$ for all $e \in E$.
\end{remark}
\begin{remark}
The subsets $S_e$ in the relaxed matching are labeled by the edge they are chosen from in order to avoid ambiguity. However if $H$ is a \textit{linear hypergraph}, that is, the edges pairwise intersect in at most one vertex, then the subsets uniquely determine the edges they belong to and therefore no labeling is necessary. Graphs and finite projective geometries (viewed as hypergraphs) are examples of linear hypergraphs. 
\end{remark}
\begin{figure}
\begin{tikzpicture}
\coordinate (v1) at (0.5,1.7);
\coordinate (v2) at (2.2,0.75);
\coordinate (v3) at (-1.5,1.5);
\coordinate (v4) at (-2.2,0);
\coordinate (v5) at (-0.5,0.5);
\coordinate (v6) at (-1.5,-1.5);
\coordinate (v7) at (-0.5,-0.65);
\coordinate (v8) at (0.8,-1.2);
\coordinate (v9) at (2.3,-0.8);
\draw (0, 0) node[inner sep=0] {\includegraphics[width=0.5\linewidth]{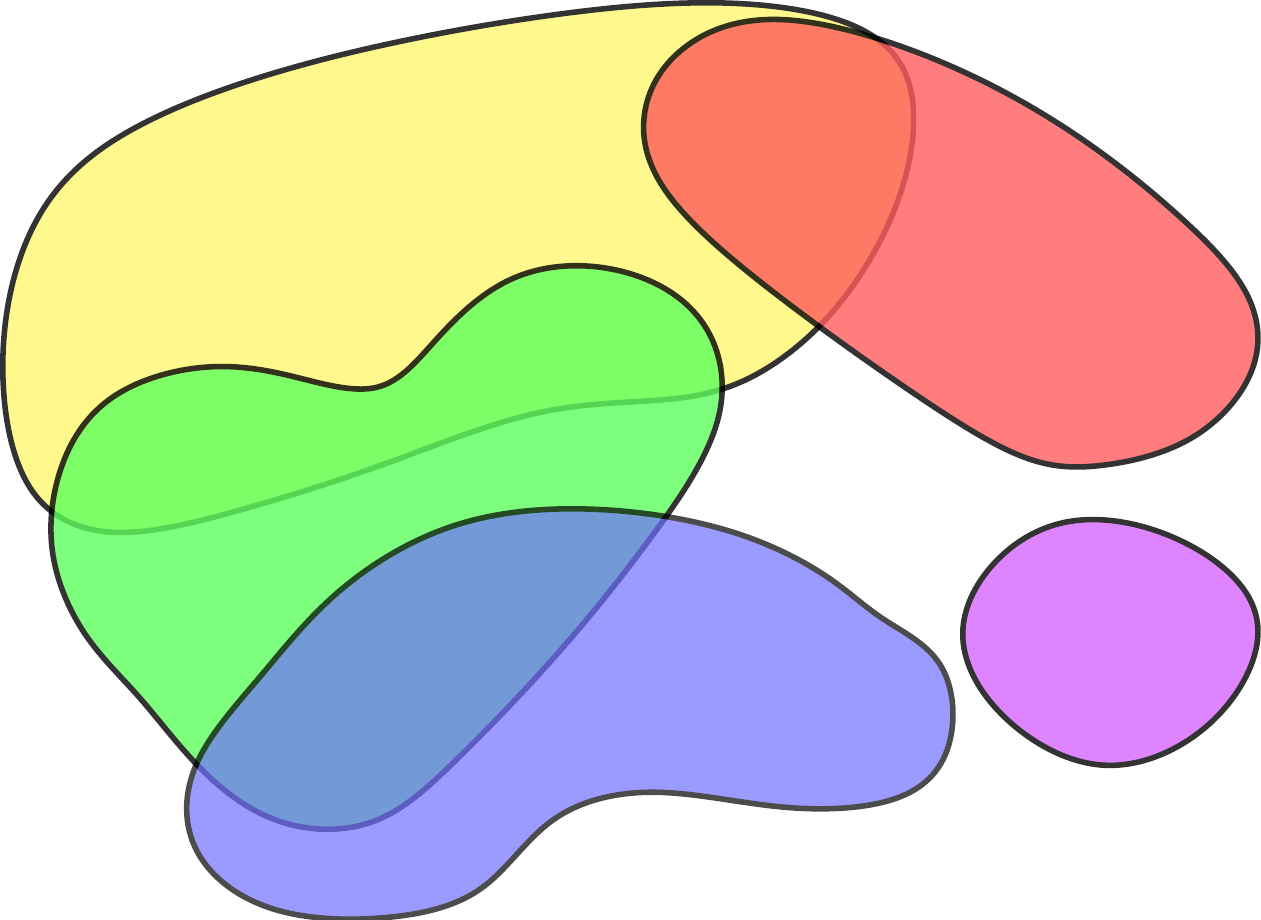}};
\draw[fill=red!80] (v1) circle (0.13cm);
\draw[fill=black] (v1) circle (0.08cm);

\draw[fill=red!80] (v2) circle (0.13cm);
\draw[fill=black] (v2) circle (0.08cm);

\draw[fill=black] (v3) circle (0.08cm);

\draw[fill=green!80] (v4) circle (0.13cm);
\draw[fill=black] (v4) circle (0.08cm);

\draw[fill=green!80] (v5) circle (0.13cm);
\draw[fill=black] (v5) circle (0.08cm);

\draw[fill=blue!80] (v6) circle (0.13cm);
\draw[fill=black] (v6) circle (0.08cm);

\draw[fill=blue!80] (v7) circle (0.13cm);
\draw[fill=black] (v7) circle (0.08cm);

\draw[fill=blue!80] (v8) circle (0.13cm);
\draw[fill=black] (v8) circle (0.08cm);

\draw[fill=black] (v9) circle (0.08cm);

\node[xshift = 0.3cm] at (v1) {$1$};
\node[xshift = 0.3cm] at (v2) {$2$};
\node[xshift = 0.3cm] at (v3) {$3$};
\node[xshift = 0.3cm] at (v4) {$4$};
\node[xshift = 0.3cm] at (v5) {$5$};
\node[xshift = 0.3cm] at (v6) {$6$};
\node[xshift = 0.3cm, yshift = 0.1cm] at (v7) {$7$};
\node[xshift = 0.3cm] at (v8) {$8$};
\node[xshift = 0.3cm] at (v9) {$9$};

\node[xshift = 1.3cm, yshift=-0.3cm] at (v1) {$e_1$};
\node[xshift = -1cm, yshift=-0.5cm] at (v3) {$e_2$};
\node[xshift = -0.1cm, yshift=-0.8cm] at (v4) {$e_3$};
\node[xshift = 1.4cm, yshift=0.1cm] at (v6) {$e_4$};
\node[xshift = 0.cm, yshift=-0.45cm] at (v9) {$e_5$};
\end{tikzpicture}
\caption{A relaxed matching $M =(S_{e_1}, S_{e_3}, S_{e_4})$ in a hypergraph $H$ with $S_{e_1} = \{1,2 \}$, $S_{e_3} = \{ 4,5\}$ and $S_{e_4} = \{ 6,7,8 \}$.}
\label{fig:relaxmatch}
\end{figure}
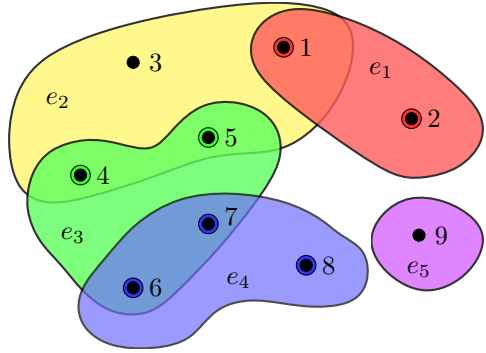

Let $V(M) \coloneqq \bigcup_{S_e \in M} S_e$ denote the set of vertices in the relaxed matching. Moreover let $m_k(M) \coloneqq |\{ S_e \in M : |S_e| = k \}|$ denote the number of subsets in the relaxed matching of size $k$.
Define the \textit{multivariate relaxed matching polynomial} of $H$ by
\[
\eta_{H}(\mathbf{x}) \coloneqq \sum_{M} (-1)^{|M|} W(M) \prod_{i \in [n]\setminus V(M)} x_i,
\]
where the sum runs over all relaxed matchings of $H$ and
\[
W(M) \coloneqq \prod_{k=1}^{n-1} k^{m_{k+1}(M)}.
\]
Let $\eta_H(x) \coloneqq \eta_H(x\mathbf{1})$ denote the \textit{univariate relaxed matching polynomial}. 
\begin{remark}
Note that $\eta_H(x) = \mu_H(x)$ if $H$ is a graph.
\end{remark}
Our aim is to prove that $\eta_{H}(\mathbf{x})$ is a stable polynomial. In fact we shall prove the stability of a more general polynomial accommodating for arbitrary degree restrictions on each vertex. 

Define a \textit{relaxed subgraph} of $H$ to be a hypergraph $K = (E(K),V(K))$ with edges $E(K) \coloneqq (S_e)_{e \in E}$ such that $E \subseteq E(H)$, $S_e \subseteq e$ and $|S_e| > 1$ for $e \in E$ with $V(K) \coloneqq \bigcup_{e \in E} S_e$. Again if $H$ is a graph, then the notion of a relaxed subgraph coincides with the conventional notion of a (edge-induced) subgraph of $H$. Let $\boldsymbol{\kappa} = (\kappa_1,\dots \kappa_n) \in \mathbb{N}^n$. Define a \textit{relaxed $\boldsymbol{\kappa}$-subgraph} of $H$ to be a relaxed subgraph $K^{\boldsymbol{\kappa}}$ of $H$ such that $\text{deg}_{K^{\boldsymbol{\kappa}}}(i) \leq \kappa_i$ for $i \in V(K^{\boldsymbol{\kappa}})$. Let $m_k(K^{\boldsymbol{\kappa}}) \coloneqq |\{ S_e \in E(K^{\boldsymbol{\kappa}}) : |S_e| = k \}|$ and let $(n)_k = n(n-1)\cdots (n-k+1)$ denote the \textit{Pochhammer symbol}. 

Define the \textit{multivariate relaxed $\boldsymbol{\kappa}$-subgraph polynomial} of $H$ by 
\[
\eta_{H}^{\boldsymbol{\kappa}}(\mathbf{x}) \coloneqq \sum_{K^{\boldsymbol{\kappa}}} (-1)^{|E(K^{\boldsymbol{\kappa}})|} W(K^{\boldsymbol{\kappa}}) \prod_{i \in [n] \setminus V(K^{\boldsymbol{\kappa}})} x_{i}^{\kappa_i - \text{deg}_{K^{\boldsymbol{\kappa}}}(i)},
\]
where the sum runs over all relaxed $\boldsymbol{\kappa}$-subgraphs $K^{\boldsymbol{\kappa}}$ of $H$ and
\[
W(K^{\boldsymbol{\kappa}}) \coloneqq \prod_{k=1}^{n-1} k^{m_{k+1}(K^{\boldsymbol{\kappa}})} \prod_{i \in V(K^{\boldsymbol{\kappa}})} (\kappa_i)_{\text{deg}_{K^{\boldsymbol{\kappa}}}(i)}.
\] 
\begin{remark}
Note that a relaxed matching in $H$ is the same as a relaxed $(1,\dots, 1)$-subgraph of $H$ and that $\eta_{H}^{(1,\dots, 1)}(\mathbf{x}) = \eta_{H}(\mathbf{x})$.
\end{remark} \noindent
In the rest of this section we will adopt the following notation,
\begin{align*}
\boldsymbol{\partial}_S \coloneqq \sum_{i \in S} \partial_i, \hspace{0.2cm} \boldsymbol{\partial}^S \coloneqq \prod_{i \in S} \partial_i, \hspace{0.2cm} \boldsymbol{\partial}^{\boldsymbol{\alpha}} \coloneqq \prod_{i=1}^n \partial_i^{\alpha_i}, 
\end{align*} \noindent
where $S \subseteq [n]$ and $\boldsymbol{\alpha} = (\alpha_i) \in \mathbb{N}^n$.

With abuse of notation we shall let the multiaffine part operator $\text{MAP}$ act analogously on polynomial spaces of differential operators as follows,
\begin{align*}
\text{MAP}:\mathbb{C}[\partial_1,\dots, \partial_n] &\to \mathbb{C}[\partial_1,\dots, \partial_n] \\ \sum_{\boldsymbol{\alpha} \in \mathbb{N}^n} a(\boldsymbol{\alpha}) \boldsymbol{\partial}^{\boldsymbol{\alpha}} &\mapsto \sum_{\boldsymbol{\alpha}: \alpha_i \leq 1, i \in [n]} a(\boldsymbol{\alpha}) \boldsymbol{\partial}^{\boldsymbol{\alpha}}.
\end{align*} \noindent
The following lemma follows from Theorem \ref{thm:liebsokal}.
\begin{lemma}\label{lem:mapstab}
If $P(\boldsymbol{\partial}) \in \mathbb{C}[\partial_1,\dots, \partial_n]$ is a linear operator such that $P(\mathbf{x}) \in \mathbb{C}[x_1,\dots, x_n]$ is stable, then $\text{MAP}\left [ P(\boldsymbol{\partial}) \right ]$ preserves stability.
\end{lemma}
\begin{proof}
Write $P(\boldsymbol{\partial}) = \sum_{\boldsymbol{\alpha} \in \mathbb{N}} a(\boldsymbol{\alpha})\boldsymbol{\partial}^{\boldsymbol{\alpha}}$. Since $\text{MAP}:\mathbb{C}[x_1,\dots,x_n] \to \mathbb{C}[x_1,\dots, x_n]$ is a stability preserver we have that $\text{MAP}\left[ \sum_{\boldsymbol{\alpha} \in \mathbb{N}^n} a(\boldsymbol{\alpha}) \mathbf{x}^{\boldsymbol{\alpha}} \right ] = \sum_{\boldsymbol{\alpha}: \alpha_i \leq 1, i \in [n]} \mathbf{x}^{\boldsymbol{\alpha}}$ is stable and hence by Theorem \ref{thm:liebsokal} that $\sum_{\boldsymbol{\alpha}: \alpha_i \leq 1, i \in [n]} \boldsymbol{\partial}^{\boldsymbol{\alpha}} = \text{MAP}\left [ \sum_{\boldsymbol{\alpha} \in \mathbb{N}^n} a(\boldsymbol{\alpha}) \boldsymbol{\partial}^{\boldsymbol{\alpha}} \right ]$ is a stability preserving linear operator.
\end{proof}
\begin{theorem} \label{thm:wksubgpolystab}
Let $H = (V(H),E(H))$ be a hypergraph and $\boldsymbol{\kappa} = (\kappa_i) \in \mathbb{N}^n$. Then the multivariate relaxed $\boldsymbol{\kappa}$-subgraph polynomial $\eta_H^{\boldsymbol{\kappa}}(\mathbf{x})$ is stable with
\[
\eta_H^{\boldsymbol{\kappa}}(\mathbf{x}) = \prod_{e \in E(H)} \text{MAP}\left [ (1- \boldsymbol{\partial}_e) \prod_{i \in e} (1+\partial_i) \right ] \mathbf{x}^{\boldsymbol{\kappa}}.
\]
\end{theorem}
\begin{proof}
Let $e \in E(H)$. Then
\begin{align*}
(1-\boldsymbol{\partial}_e) \prod_{i \in e}(1+\partial_i)  &= (1-\boldsymbol{\partial}_e)\left (1 + \sum_{\emptyset \neq S \subseteq e} \boldsymbol{\partial}^S \right )  \\ &= 1 + \sum_{\substack{\emptyset \neq S \subseteq e\\ |S| > 1}} \boldsymbol{\partial}^S - \sum_{\emptyset \neq S \subseteq e} \boldsymbol{\partial}_e \boldsymbol{\partial}^S \\ &= 1  +  \sum_{\substack{\emptyset \neq S \subseteq e\\ |S| > 1}} \boldsymbol{\partial}^S - \sum_{i \in e} \left ( \sum_{\substack{\emptyset \neq S \subseteq e\\ i \in S}} \partial_i\boldsymbol{\partial}^S + \sum_{\substack{\emptyset \neq S \subseteq e\\ i \not \in S}} \boldsymbol{\partial}^{S \cup i} \right ) \\ &= 1 - \sum_{\substack{\emptyset \neq S \subseteq e \\ |S|>1}} (|S|-1)\boldsymbol{\partial}^S - \sum_{i \in e} \sum_{\substack{\emptyset \neq S \subseteq e\\ i\in S}} \partial_i\boldsymbol{\partial}^S.
\end{align*} \noindent
Thus since
\[
\left (1-\sum_{i\in e}x_i \right )\prod_{i \in e} (1+x_i),
\]
is a stable polynomial, being a product of stable linear factors, it follows by Lemma \ref{lem:mapstab} that
\[
\text{MAP} \left [ (1-\boldsymbol{\partial}_e)\prod_{i \in e}(1+\partial_i) \right ] = 1-\sum_{\substack{\emptyset \neq S \subseteq e \\ |S|>1}} (|S|-1)\boldsymbol{\partial}^S, 
\] 
is stability preserving. Hence
\begin{align*}
\prod_{e \in E(H)} \text{MAP}\left [ (1- \boldsymbol{\partial}_e) \prod_{i \in e} (1+\partial_i) \right ] \mathbf{x}^{\boldsymbol{\kappa}} &= \prod_{e \in E(H)} \left ( 1-\sum_{\substack{\emptyset \neq S \subseteq e \\ |S|>1}} (|S|-1)\boldsymbol{\partial}^S \right )\mathbf{x}^{\boldsymbol{\kappa}}  \\ &= \sum_{E \subseteq E(H)} (-1)^{|E|} \sum_{\substack{(S_e)_{e \in E} \\ S_e \subseteq e \\ |S_e|>1}} \prod_{e \in E} (|S_e|-1) \boldsymbol{\partial}^{S_e} \mathbf{x}^{\boldsymbol{\kappa}} \\ &= \sum_{K^{\boldsymbol{\kappa}}} (-1)^{|E(K^{\boldsymbol{\kappa}})|} W(K^{\boldsymbol{\kappa}}) \prod_{i \in [n]\setminus V(K^{\boldsymbol{\kappa}})} x_i^{\kappa_i- \text{deg}_{K^{\boldsymbol{\kappa}}}(i)} \\ &= \eta_{H}^{\boldsymbol{\kappa}}(\mathbf{x}),
\end{align*} \noindent
is a stable polynomial.
\end{proof}
The following corollary is immediate from Theorem \ref{thm:wksubgpolystab}.
\begin{corollary}
The multivariate relaxed matching polynomial $\eta_H(\mathbf{x})$ is stable with
\[
\eta_H(\mathbf{x}) = \prod_{e \in E(H)} (1-\boldsymbol{\partial}_e) \prod_{i=1}^n (1+\partial_i)^{\text{deg}_H(i)} \mathbf{x}^{\mathbf{1}}.
\]
In particular the univariate relaxed matching polynomial 
\[
\eta_H(x) = \sum_{M} (-1)^{|M|} W(M) x^{n-|V(M)|},
\] 
is a real-rooted polynomial for any hypergraph $H$.
\end{corollary} \noindent
Below follows a generalization of the standard identities for the multivariate matching polynomial $\mu_G(\mathbf{x})$. Let $i \in V(H)$. Recall that the \textit{(weak) vertex-deletion} $H \setminus i$ is the hypergraph with vertex set $V(H)\setminus i$ and edges $\{ e \cap (V(H) \setminus i) : e \in E(H)\}$.  Let $e \in E(H)$. The \textit{edge-deletion} $H \setminus e$ is the subhypergraph of $H$ with vertex set $V(H)$ and edges $E(H)\setminus e$.
Let $I_H(i) \coloneqq \{ e \in E(H): i \in e\}$ denote the \textit{incidence set} of $i \in V(H)$. The following identities are straightforward to verify.
\begin{proposition}
Let $H=  (V(H),E(H))$ be a hypergraph, $i \in V(H)$ and $e \in E(H)$. Then $\eta_H(\mathbf{x})$ satisfies the following identities:
\begin{enumerate}
\item $\displaystyle \eta_H(\mathbf{x}) = \eta_{H\setminus e}(\mathbf{x}) - \sum_{\substack{S \subseteq e\\ |S|>1}} (|S|-1)\eta_{(H\setminus e)\setminus S}(\mathbf{x})$,
\item $\displaystyle \eta_H(\mathbf{x}) = x_i\eta_{H\setminus i}(\mathbf{x}) - \sum_{e \in I_H(i)} \sum_{\substack{S \subseteq e \\ i \in S \\ |S|> 1}} (|S|-1) \eta_{(H\setminus e) \setminus S}(\mathbf{x})$,
\item $\displaystyle \eta_{H_1 \sqcup H_2}(\mathbf{x}) = \eta_{H_1}(\mathbf{x}) \eta_{H_2}(\mathbf{x})$,
\item $\displaystyle \partial_i \eta_H(\mathbf{x}) = \eta_{H\setminus i}(\mathbf{x})$.
\end{enumerate}
\end{proposition} \noindent
It would be interesting to understand what parts of the matching theory for graphs can be extended to relaxed matchings.
\newline \newline
\textbf{Acknowledgements:} The author would like to thank Petter Br\"and\'en for helpful discussions and the anonymous referee for pointing out an issue with Theorem \ref{thm:multidmatchpoly} in an earlier version of this paper.

\end{document}